\documentclass[a4paper,12pt]{article}

\usepackage{a4wide}
\usepackage{amsmath, amssymb, amsthm}
\usepackage{amsfonts}
\usepackage{latexsym}
\usepackage{eucal}
\usepackage{xcolor}
\usepackage{enumerate}
\usepackage{graphicx}

\author{
  Tom\'a\v s Kaiser, Robert Lukot\!'ka, Edita M\' a\v cajov\' a, Edita Rollov\' a
}
\title{\textbf{Shorter signed circuit covers of graphs}}
\date{}

\newtheorem{theorem}{Theorem}[section]

\newtheorem{lemma}[theorem]{Lemma}
\newtheorem{corollary}[theorem]{Corollary}

\newtheorem{conjecture}[theorem]{Conjecture}

\newtheorem{observation}[theorem]{Observation}
\newtheorem{claim}{Claim}

\newcommand{\claimproofend}{\hspace*{.1mm}\hspace{\fill}}

\newcommand{\negat}[1]{\eps(#1)} \newcommand{\negatx}[1]{\eps'(#1)}
\newenvironment{xcase}[1]{\par\medskip\noindent\textbf{\emph{#1:}}\itshape}{\normalfont\par\smallskip}

\newcommand{\Setx}[1]{\left\{#1\right\}}
\newcommand{\size}[1]{|#1|}

\newcommand{\CC}{\mathcal C}
\newcommand{\KK}{\mathcal K}

\let\epsilon=\varepsilon
\let\eps=\varepsilon
\let\phi = \varphi

\begin{document}

\maketitle

\abstract{A signed circuit is a minimal signed graph (with respect to
  inclusion) that admits a nowhere-zero flow.  We show that each
  flow-admissible signed graph on $m$ edges can be covered by signed
  circuits of total length at most $(3+2/3)\cdot m$, improving a
  recent result of Cheng et al. [manuscript, 2015].  To obtain this
  improvement we prove several results on signed circuit covers of
  trees of Eulerian graphs, which are connected signed graphs such
  that removing all bridges results in a collection of Eulerian
  graphs.}


\section{Introduction}
A \emph{circuit cover} of a bridgeless graph $G$ is a collection of circuits such that each
edge of $G$ belongs to at least one of them. One of the most studied problems concerning circuit covers is finding a circuit cover of the graph with small total length. A conjecture by Alon and Tarsi
\cite{AT} bounds the length of the shortest circuit cover of a bridgeless graph from above.

\begin{conjecture}[Short Cycle Cover conjecture]
  Every bridgeless graph $G$ has a circuit cover
  of total length at most $7/5 \cdot |E(G)|$.
\end{conjecture}
\noindent The best general upper bound on the length of a shortest cycle cover is a result obtained independently by Bermond, Jackson, and Jaeger~\cite{BJJ} and by Alon and Tarsi~\cite{AT}.

\begin{theorem}\emph{\cite{AT,BJJ}}\label{BJJ}
  If $G$ is a bridgeless graph, then it admits a circuit cover of total length at most $5/3\cdot |E(G)|$.
\end{theorem}
\noindent The Short Cycle Cover Conjecture has connections to many well-known conjectures:
for instance, it is implied by the Petersen Flow Conjecture (alternatively known as the Petersen Colouring Conjecture) of Jaeger \cite[Section~7]{jaeger2},
while it implies the Cycle Double Cover Conjecture \cite{jamshy}.

\noindent In parallel to the classical graph theory there is a
fast-growing theory of \emph{signed graphs}, graphs where each edge
has a positive or a negative sign.  More formally, a signed graph
$(G,\sigma)$ is a graph $G$ endowed with a function $\sigma$ called
\emph{signature} which assigns values either $1$ (\emph{positive
  edges}) or $-1$ (\emph{negative edges}) to the edges. A circuit $C$
in $G$ is \emph{balanced} if the number of negative edges in $C$ is
even, and \emph{unbalanced} otherwise. Two signed graphs
$(G,\sigma_1)$ and $(G,\sigma_2)$ are \emph{equivalent} if they have
the same set of balanced circuits. Equivalent signed graphs are
considered to be the same, and one can freely replace the signature of
a signed graph with the signature of any equivalent signed graph. We
say that a signature $\sigma$ is a \emph{minimum signature} of
$(G,\sigma)$ if there is no signed graph equivalent to $(G,\sigma)$
with fewer negative edges.

Many fundamental concepts in graphs, such as colourings, flows,
homomorphisms, or surface embeddings found their analogues in signed
graphs~\cite{zaslav, bouchet, RRS, SS}.  In this paper we will study a
problem of short circuit covers for signed graphs introduced by
M\'{a}\v{c}ajov\'{a} et al. in \cite{MRRS}. We refer the reader
to~\cite{KLR} for a survey on circuit covers and nowhere-zero flows in
signed graphs, and to~\cite{Zs} for an extensive bibliography on
signed graphs.

According to~\cite{MRRS}, a \emph{signed circuit} of a signed graph is
one of the following subgraphs: (1) a balanced circuit, (2) the union
of two unbalanced circuits which meet at a single vertex --- a
\emph{short barbell}, or (3) the union of two disjoint unbalanced
circuits with a path which meets the circuits only at it ends---a
\emph{long barbell}. A \emph{barbell} is either a short or a long
barbell.

The definition of signed circuit is chosen {{so}} that several important
correspondences are preserved. A signed circuit in a signed graph $G$
forms a signed circuit in the signed graphic matroid $\mathcal{M}(G)$
and vice versa \cite[Theorem~5.1]{zaslav1}. Moreover, a signed graph
is \textit{flow-admissible}, that is it admits a signed nowhere-zero
flow, if and only if every edge of the signed graph belongs to a
signed circuit \cite[Proposition~3.1]{bouchet}. Furthermore,
signed circuits are the minimal graphs that are flow-admissible, a
property that circuits satisfy in ordinary graphs.

Let $\CC$ be a collection of subgraphs of a signed graph $G$. We say
that $\CC$ is a \emph{cover} if each edge of $G$ is \emph{covered by
  $\CC$} (i.e. it is contained in at least one subgraph from $\CC$),
and it is a \emph{signed circuit cover} if each element of $\CC$ is a
signed circuit. The \emph{length} of $\CC$ is the sum
$\sum_{F\in\CC} \size{E(F)}$. The \emph{width} $\omega_\CC(e)$ of an
edge $e$ with respect to $\CC$ is defined as the number of subgraphs
from $\CC$ containing $e$. If $\omega_{\CC}(e)=k$, then we also say
that $e$ is \emph{covered $k$ times} by $\CC$. The \emph{width} of the
collection $\CC$, denoted by $\omega(\CC)$, is the maximum width of an
edge with respect to $\CC$ over all edges of $G$.

If $\CC_1,\dots,\CC_\ell$ are collections of subgraphs of $G$, then
the \emph{total length} of $\Setx{\CC_1,\dots,\CC_\ell}$ is the sum of
lengths of all the $\CC_i$ ($i\in\{1,\dots,\ell\}$). The \emph{total width}
of an edge $e$ with respect to the collection
$\Setx{\CC_1,\dots,\CC_\ell}$ is defined as
\begin{equation*}
  \omega_{\CC_1,\dots,\CC_\ell}(e) = \sum_{i=1}^\ell \omega_{\CC_i}(e).
\end{equation*}
Again, the \emph{total width} of the collection $\Setx{\CC_1,\dots,\CC_\ell}$,
denoted by $\omega(\CC_1,\dots,\CC_\ell)$, is the maximum width of an edge
with respect to this collection over all edges of $G$.

M\'a\v cajov\'a et al.~\cite{MRRS} proved that every flow-admissible
signed graph with $m$ edges has a signed circuit cover of length at
most $11\cdot m$.  Cheng, Lu, Luo and Zhang~\cite{CLLZ} announced in
2015 an improvement of the bound to
$14/3\cdot m - 5/3 \cdot \epsilon_N -4$ for flow-admissible signed
graphs with $\epsilon_N$ negative edges in a minimum signature. A
refinement of their proof idea leads to the following improvement.

\begin{theorem}\label{thm:main}
  Let $(G,\sigma)$ be a flow-admissible signed graph with $m$ edges and let $\epsilon_N$ be the number of negative edges in a minimum signature of $(G,\sigma)$.
  Then there is a signed circuit cover of $(G,\sigma)$ of length at most $11/3\cdot m-5/3 \cdot \epsilon_N$.
\end{theorem}

\noindent In Section~\ref{sec5} we provide several variants of Theorem~\ref{thm:main}
comparable to the results provided by Cheng et al.~\cite{CLLZ}. We note that for bridgeless cubic signed graphs, it is possible to obtain even better bounds on the length of a short signed circuit cover as suggested by a recently announced result by Wu and Ye \cite{Wu}.

Despite these improvements, the bounds on the length of the shortest
signed circuit cover are probably far from tight. Indeed, we are not
aware of any flow-admissible signed graph $(G,\sigma)$ with a shortest
cycle cover of length exceeding $5/3\cdot |E(G)|$. An example that
attains this value is the Petersen graph with five negative edges
forming a 5-cycle~\cite{MRRS}. 


\section{Covering auxiliary structures}
In this section we will introduce the necessary notions and prove
several key lemmas. Graphs in this paper may have parallel edges and
loops. For convenience, we define the \emph{subtraction} of two
subgraphs $A_1$ and $A_2$ of a graph, denoted by $A_1-A_2$, as
follows: from $A_1$, we remove all edges that are contained in $A_2$
and delete isolated vertices.  A \emph{cut-vertex} is a vertex whose
deletion increases the number of components of a graph. Note that a
vertex incident with a loop is not necessarily a cut-vertex. A graph
is \emph{2-connected} if it has no cut-vertex. Note that graphs
consisting either of a vertex with loops incident to it or of an edge
with loops incident to its end-vertices are both 2-connected.  A
\emph{cycle} is a graph with all degrees even. A connected cycle is an
\emph{Eulerian graph} (in particular, the graph with a single vertex
is an Eulerian graph, and we refer to it as \emph{trivial}). A
\emph{circuit} is a connected 2-regular graph (note that a loop is a
circuit). Recall that a \emph{signed circuit} is a signed graph that
is either a balanced circuit or a barbell.

Let $H$ be a {{connected}} signed graph. Delete all bridges of $H$ to obtain
$H'$. We let $\negat H$ denote the number of negative edges in $H$, and
let $\negatx H$ denote the number of negative edges in $H'$.

If each component of $H'$ is an Eulerian graph, then $H$ is a
\emph{tree of Eulerian graphs}. In the further text the notion of a balloon will be crucial. A \emph{balloon of $H$} (or simply a \emph{balloon} if $H$ is clear from the context) is either a negative loop of $H'$ or a component of $H'$ from which all the loops have been deleted.
If each balloon of $H'$ is either an isolated vertex or a
circuit, then $H$ is a \emph{tree of circuits}. Note that a tree of
circuits is a special case of a tree of Eulerian graphs.

Let $B$ be a balloon of a tree of Eulerian graphs $H$. We say that $B$
is \emph{trivial} if $B$ is a single vertex, and \emph{non-trivial}
otherwise. The \emph{valency of $B$} is $1$ if $B$ is a negative loop,
and otherwise it is the number of its incident bridges and loops. The
balloon $B$ is a \emph{leaf balloon} if its valency is $1$, and an
\emph{inner balloon} otherwise. In the case that $B$ is a circuit, we
use the terms \emph{leaf circuit} or \emph{inner circuit}. The balloon
$B$ is \emph{even} or \emph{odd} according to $\negat B$ being even or
odd, respectively. An \emph{endblock} of $H$ is an inclusion-wise maximal subgraph of $H$ which contains at most one cut-vertex and is different from a loop.

The following definitions and two lemmas are based on the ideas
of~\cite{MRS1} and \cite{MRS2}. Suppose for now that $H$ is a tree of
circuits.  Let $c(H)$ and $u(H)$ denote the number of circuits and the
number of unbalanced circuits of $H$, respectively. Suppose that $H$
is a tree of circuits such that every leaf
circuit is unbalanced. Note that if $H$ has at least two leaf circuits, it is flow-admissible.  Let
$\vec{H}$ be a digraph obtained from $H$ by the following operations:
\begin{itemize}
\item orienting the edges of each circuit in $H$ in such a way that it
  becomes a directed circuit,
\item replacing all bridges by directed 2-circuits.
\end{itemize}
Since $\vec{H}$ is an Eulerian digraph, it admits an Eulerian trail
$W$. Traverse $W$ starting from some vertex of $\vec{H}$ and enumerate
all circuits of $H$ as
\begin{equation*}
  C_0,C_1,\dots,C_{k-1},
\end{equation*}
in the order in which $W$ visits them for the first time. (Thus, each
circuit of $H$ appears exactly once in the sequence.) Let us say that
two leaf circuits $C_i, C_j$ are \emph{consecutive leaf circuits} if
none of $C_{i+1},\dots,C_{j-1}$ is a leaf circuit (where the counting
is modulo $k$). \emph{Consecutive unbalanced circuits} are defined in
an analogous way.
For any two indices $s,t \in \Setx{0,\dots,k-1}$, let $W_{st}$ be a
shortest subtrail of $W$ starting at a vertex of $C_s$ and ending at a
vertex of $C_t$.
Define $B_{st}$ to be the barbell in $H$ obtained as the union of
$C_s$, $C_t$ and the path corresponding to $W_{st}$.

\begin{lemma}\label{l:leaf}
  Let $H$ be a tree of circuits with $u(H) \neq 1$ such that every
  leaf circuit is unbalanced. Then $H$ admits a signed circuit cover
  $\CC$ such that the width of the edges of leaf circuits and bridges
  with respect to $\CC$ is $2$, and the width of the edges of inner
  circuits with respect to $\CC$ is~$1$.
\end{lemma}

\begin{proof}
  If $H$ is balanced, then since each leaf is unbalanced, we have $c(H)=1$.
  As $u(H) \neq 1$ the circuit is balanced and we set $\CC=\{H\}$.
  Thus, we may assume that $u(H)\geq 2$.

  Define $\CC$ as the collection of all barbells $B_{st}$ such that
  $C_s$ and $C_t$ are consecutive leaf circuits. Clearly, edges of
  leaf circuits and bridges have width $2$ with respect to $\CC$
  while edges of an inner {{circuit}} have width $1$ with respect to $\CC$.
\end{proof}

A collection $\CC$ of subgraphs of a signed graph $H$
is a \emph{weak signed circuit cover} if all the elements of
$\CC$ are signed circuits and each non-bridge edge of $H$ is covered by $\CC$.

\begin{lemma}
  \label{l:43}
  Let $H$ be a tree of circuits with $u(H)$ even. Then $H$ admits
  three weak signed circuit covers $\CC_1$, $\CC_2$, $\CC_3$ of total
  width at most $4$ such that $\CC_1$ covers all negative loops
  exactly twice.
\end{lemma}

\begin{proof}
  We proceed by induction on the number of edges of $H$. If $H$ is
  balanced, then we set $\CC_1, \CC_2$, and $\CC_3$ to be the set of circuits of $H$.
  If there is a bridge $e$ of $H$ such that some component of $H-e$ is
  balanced, then we cover each component of $H-e$ by the induction
  hypothesis and define each $\CC_i$ as the union of the two obtained respective covers.
  Therefore, we may assume that every leaf balloon of $H$ is
  unbalanced. Note that $u(H)\geq 2$.

  Let us define the signed circuit covers $\CC_1$, $\CC_2$ and
  $\CC_3$.  Let $\CC_1$ be the weak signed circuit cover defined in
  Lemma~\ref{l:leaf}. The other two weak signed circuit covers $\CC_2$
  and $\CC_3$ are defined as follows.

  Let $\KK$ be the collection of balanced circuits of $H$ and suppose
  that the unbalanced circuits of $H$ are $C_{i_0},\dots,C_{i_{u-1}}$,
  where $u=u(H)$ and $0 \leq i_0 < \dots < i_{u-1} \leq u-1$. If
  $B^*_{rs}$ denotes the barbell $B_{i_ri_s}$, then
  \begin{align*}
    \CC_2 &= \Setx{B^*_{01}, B^*_{23}, \dots, B^*_{u-2,u-1}}
            \cup \KK,\\
    \CC_3 &= \Setx{B^*_{12}, B^*_{34}, \dots, B^*_{u-1,0}} \cup \KK.
  \end{align*}

  Let us check that the total width of $\CC_1$, $\CC_2$ and $\CC_3$ is
  at most 4. If $e$ is an edge of a leaf circuit, it is covered twice
  by $\CC_1$, once by $\CC_2$ and once by $\CC_3$. If $e$ is a bridge,
  it is covered twice by $\CC_1$ and twice by $\CC_2\cup\CC_3$,
  because the paths $W_{i,i+1}$ ($0\leq i\leq u-1$) are edge-disjoint
  (in $\vec{H}$). For the same reason, an edge $e$ of an inner circuit
  is covered once by some $W_{i,i+1}$ and once by a circuit of each of
  $\CC_2$ and $\CC_3$. Since by Lemma~\ref{l:leaf} the edges of inner
  circuits have width $1$ in $\CC_1$, we have found the required
  covers.
\end{proof}

We now proceed to trees of Eulerian graphs.

\begin{lemma}\label{l:rozklad}
  Let $A$ be a $2$-connected non-trivial signed Eulerian graph without
  ne\-ga\-ti\-ve loops, and let $v$ be a vertex of $A$. Then either
  $A$ is a circuit, or it can be edge-decomposed into two nontrivial
  Eulerian graphs $A_1$ and $A_2$ such that $A_1$ contains $v$ and
  $\negat{A_2}$ is even.
\end{lemma}

\begin{proof}
  Since $A$ is Eulerian, there is a circuit $C$ containing $v$. Let
  $A'$ be a component of $A-C$ (recall that $A-C$ has no isolated
  vertices due to our definition of graph subtraction). If $\negat{A'}$
  is even, $A$ can be decomposed into $A'$ and $A-A'$.

  Suppose then that $\negat{A'}$ is odd. By the assumption of the lemma,
  $A'$ is not a loop, and by the 2-connectivity of $A$,
  $|V(C) \cap V(A')|\geq 2$. Let $v_1$ and $v_2$ be vertices of
  $C\cap A'$. Let $P$ be a $v_1$-$v_2$-path of $C$ that does not
  contain $v$ as an internal vertex. Since $A$ is Eulerian, $A'$ is
  Eulerian as well, so it admits a closed Eulerian trail $W$. The
  trail $W$ consists of two $v_1$-$v_2$ trails, $W_1$ and $W_2$. Since
  $\negat{A'}$ is odd, we may assume without loss of generality that
  $\negat{W_1}$ has the same parity as $\negat P$. The sought
  decomposition is $W_1\cup P$ and $A-(W_1\cup P)$.
\end{proof}

To deal with loops we define two operations on Eulerian graphs:
compression and decompression.  Let $(G,\sigma)$ be a signed Eulerian
graph with a non-loop edge, and let $L$ be a set of negative loops of
$G$. Let $f:\,L \to (E(G)-L)$ be a function that assigns to each loop
edge an adjacent non-loop edge of $G$. We say that $(G^*,\sigma^*)$ is
a \emph{compression of $(G,\sigma)$ with respect to $f$} if $G^*=G-L$
and $\sigma^*(e):=\sigma(e)\cdot(-1)^{|f^{-1}(e)|}$ for every edge $e$
of $G^*$. Let $(G_0^*,\sigma^*)$ be a subgraph of
$(G^*,\sigma^*)$. (We take the liberty of using the same symbol for
the induced signature of the subgraph.) The \emph{decompression of
  $(G_0^*,\sigma^*)$ with respect to $f$} is the subgraph
$(G_0,\sigma)$ of $(G,\sigma)$ induced by the edges
$E(G^*_0) \cup f^{-1}(E(G^*_0))$. Note the following.

\begin{observation}\label{obs:compress}
  Compression and decompression preserve the parity of the number of negative edges. If two subgraphs partition the compressed graph, then
  the decompressed subgraphs partition the original graph. If the original graph is $2$-connected, then so is the compressed graph.
\end{observation}

The following lemma and its corollary generalise the result of \cite{MRS1} and \cite{MRS2} stating that every signed Eulerian graph $G$ with even number of negative edges admits a signed circuit cover of total length $4/3\cdot |E(G)|$.

\begin{lemma}\label{l:eventree}
  Let $H$ be a tree of Eulerian graphs such that $\negatx H$ is
  even. Then $H$ admits three weak signed circuit covers $\CC_1$,
  $\CC_2$, $\CC_3$ of total width at most $4$ such that
  $\omega(\CC_i)\leq 2$ for each $i\in\{1,2,3\}$ and $\CC_1$ covers all
  negative loops exactly twice.
\end{lemma}

\begin{proof}
  Let $H$ be a counterexample to the lemma with minimum number of
  edges. Clearly, $H$ has no balanced loops.

  Assume that $H$ is not $2$-connected, and consider an end-block $H_1$
  of $H$ incident with a cut-vertex $v$. Recall that by the definition of an end-block, $H_1$ is different from a loop. Let $H_2$ be obtained from
  $H$ by removing $V(H_1)-v$.

  Assume first that $\negatx{H_1}$ is even. Then since $\negatx H$ is
  even, so is $\negatx{H_2}$. Since both $H_1$ and $H_2$ contain fewer
  edges than $H$, they admit the weak signed circuit covers described
  in the statement of the lemma. The combination of these covers gives
  the sought weak signed circuit covers of $H$, a contradiction.

  Thus, both $\negatx{H_1}$ and $\negatx{H_2}$ are odd (hence each of
  $H_1,H_2$ has at least two edges). For $i\in\{1,2\}$, we add to $H_i$ a
  negative loop $e_i$ at $v$ to obtain $H'_i$. Both $H'_1$ and
  $H'_2$ contain fewer edges than $H$, and hence they admit the weak
  signed circuit covers with the requested properties. Note that $e_1$
  and $e_2$ must be covered by barbells, and each of them is covered
  twice by the first cover and once by the two other covers.
{{To obtain the sought weak signed circuit covers of $H$, we combine the found signed circuit covers of $H'_1$ and $H'_2$ by merging the barbells
  containing loops $e_1$ and $e_2$.
  This contradicts the
  minimality of $H$.}}

  It follows that $H$ is $2$-connected. There are several
  possibilities: either $H$ is a vertex with loops incident with it,
  or $H$ consists of an edge and loops incident with the end-vertices
  of the edge, or $H$ is bridgeless (on at least two vertices). In the
  first two cases, it is easy to find the weak signed circuit covers
  directly. Thus, it suffices to consider the third case, where $H$
  consists of one non-loop non-trivial balloon and possibly some {{unbalanced}} loops
  incident with it.

  Let $H^*$ be a compression of $H$ with respect to an arbitrary
  assignment $f$. Let $v$ be an arbitrary vertex of $H^*$. By
  Lemma~\ref{l:rozklad}, either $H^*$ is a circuit or it can be
  decomposed into two Eulerian graphs each having even number of
  negative edges. In the first case, Lemma~\ref{l:43} provides the
  desired covers. Assume thus that $H^*$ decomposes into nontrivial
  Eulerian graphs $H^*_1$ and $H^*_2$ with $\negat{H^*_1}$ and
  $\negat{H^*_2}$ even.  Let $H_1$ and $H_2$ be decompressions of
  $H^*_1$ and $H^*_2$.  By Observation~\ref{obs:compress},
  $\negat{H_1}$ and $\negat{H_2}$ is even and each of $H_1$ and $H_2$ has
  fewer edges than $H$. Hence they admit the sought weak signed
  circuit covers. Since $H_1$ and $H_2$ partition $H$ by
  Observation~\ref{obs:compress}, the covers of $H_1$ and $H_2$
  combine to form the required covers of $H$. This contradiction
  concludes the proof.
\end{proof}

\begin{corollary}\label{cor:eventree}
  Let $H$ be a tree of Eulerian graphs with $\negatx H$ even. Then $H$
  admits a weak signed circuit cover of length $4/3\cdot |E(H)|$.
\end{corollary}

\begin{lemma}\label{l:gentree}
  Let $H$ be a tree of Eulerian graphs such that it has at least two
  leaf balloons, and each leaf balloon is odd. Then $H$ admits a
  signed circuit cover of width at most $2$ that covers every loop of
  $H$ exactly twice.
\end{lemma}

\begin{proof}
  Let $H$ be a counterexample to the lemma with the minimum number of
  edges. Clearly, $H$ has no balanced loops.

  We prove that $H$ is $2$-connected. Suppose that this is not the
  case. Consider an end-block $H_1$ of $H$ incident with a cut-vertex
  $v$, and let $H_2$ be obtained from $H$ by removing all the vertices
  of $H_1$ except $v$.

  Assume that $H_1$ contains only one balloon and that balloon is even
  (hence $H_1$ has no loops). By Lemma~\ref{l:eventree}, $H_1$ admits
  a signed circuit cover of width at most $2$. Any leaf balloon of
  $H_2$ is either a leaf balloon of $H$, or it is obtained from a leaf
  balloon of $H$ by deleting $H_1$. By the assumption that
  $\negat{H_1}$ is even, every leaf balloon of $H_2$ is odd, and there
  are at least two of them. Since $H_2$ contains fewer edges than $H$,
  it admits the desired signed circuit cover and so does $H$, a
  contradiction.

  Thus, either $H_1$ is an odd balloon, or $H_1$ contains at least two
  balloons, at least one of which is a leaf balloon of $H$, and hence
  it is odd. We conclude that $H_1$ contains an odd balloon. By
  symmetry, $H_2$ contains an odd balloon. Note that in each of $H_1$
  and $H_2$, every leaf balloon (except possibly for one containing
  $v$) is a leaf balloon of $H$. For $i\in\{1,2\}$, we define $H'_i$ as the
  graph obtained from $H_i$ by adding a negative loop $e_i$ at
  $v$. For $i\in\{1,2\}$, each leaf balloon of $H'_i$ is odd and there are
  at least two of them. Since $H_1$ and $H_2$ contain fewer edges than
  $H$, they admit signed circuit covers described in the
  statement. Note that $e_1$ and $e_2$ must be covered by barbells,
  and each of them is covered exactly twice. {{To obtain the sought signed circuit
  cover of $H$, we combine the found signed circuit covers of $H_1$ and $H_2$ by merging barbells
  containing loops $e_1$ and $e_2$. This contradicts the choice of $H$.}}

  Thus, $H$ is 2-connected. That is, either $H$ is a vertex with loops
  incident with it, or $H$ consists of an edge and loops incident with
  end-vertices of the edge, or $H$ is bridgeless (on at least two
  vertices). In the first two cases, it is easy to find the desired
  signed circuit cover directly. In the third case, $H$ consists of
  one non-loop balloon and possibly {{unbalanced}} loops incident with vertices of
  the balloon. As $H$ has at least two leaf balloons, it contains at
  least one loop. If $H$ has exactly one negative loop, then the
  non-loop balloon is odd and the cover $\CC_1$ from
  Lemma~\ref{l:eventree} provides a contradiction. Thus, $H$ has at
  least two loops. Denote by $e_1$ and $e_2$ two arbitrary distinct
  loops of $H$, and by $v_1$ and $v_2$ the end-vertices of $e_1$ and
  $e_2$, respectively; it may happen that $v_1=v_2$.

  Let $H^*$ be a compression of $H-\{e_1,e_2\}$ with respect to an
  arbitrary assignment $f$. If we partition $H^*$ into Eulerian
  subgraphs in such a way that one of them, denoted by $H^*_S$,
  contains $v_1$ and $v_2$ and {{every other either contains
   even number of negative edges or its decompression contains}} at least two loops, then we can produce
  a cover of $H$ as follows.  We take the decompression of all
  subgraphs. The union of the decompression of $H^*_S$ with $e_1$ and
  $e_2$, denoted by $H_S$, has fewer edges than $H$ and as $H$ is a
  smallest counterexample to the lemma, $H_S$ can be covered so that
  the conditions of the lemma are satisfied.  The decompressions of
  other subgraphs either satisfy assumptions of Lemma~\ref{l:eventree}
  (those subgraphs that have even number of negative edges), or they
  have at least two loops and satisfy assumptions of this lemma while
  having fewer edges than $H$. Thus we can cover them as required. A
  cover of $H$ is then obtained as the union of the covers of the
  subgraphs under consideration. Therefore, finding such a partition
  of $H^*$ provides a contradiction with the choice of $H$.

  If $H^*$ is a circuit, then Lemma~\ref{l:leaf} provides the desired
  cover. If $v_1=v_2$, then by Lemma~\ref{l:rozklad}, $H^*$ can be
  partitioned into Eulerian subgraphs $H_S^*$ and $H_2^*$, such that
  $H_S^*$ contains $v_1$ and $\negat{H_2^*}$ is even, a
  contradiction. Thus, $v_1\neq v_2$.

  Since $H^*$ is 2-connected, there exist two internally
  vertex-disjoint $v_1$-$v_2$-paths, $P_1$ and $P_2$. As $H^*$ is not
  a circuit, $H^*-P_1-P_2$ is non-empty. Let $A$ be a component of
  $H^*- P_1-P_2$. Assume that $A$ contains a cut-vertex $r$ and let $A_1$ and $A_2$ be subgraphs of $A$ such that $A_1\cup A_2=A$ and $A_1\cap A_2=\{r\}$. Since $H^*$ is 2-connected, both $A_1$ and $A_2$ intersect $H^*-A$. If one of $A_1$ and $A_2$, say $A_1$, is even, then $H^*$ can be decomposed into $A_1$ and $H^*-A_1$ which leads to a contradiction. Thus both $A_1$ and $A_2$ are odd which implies that $A$ is even and so $H^*$ can be decomposed into $A$ and $H^*-A$ which are smaller graphs that fulfill the conditions of the lemma, a contradiction. Therefore we can assume that $A$ is 2-connected.  

  Let $w \in V(A) \cap V(P_1 \cup P_2)$.  By
  Lemma~\ref{l:rozklad}, either $A$ is a circuit, or it can be
  decomposed into Eulerian subgraphs $H^*_1$ and $H^*_2$ such that
  $H^*_1$ contains $w$ and $\negat{H^*_2}$ is even. In the latter
  case, the decomposition of $H^*$ into $H^*_S=H^*-H^*_2$ and $H^*_2$
  provides a contradiction. Therefore $A$ is a circuit. If $A$ is
  balanced, then we get a contradiction by decomposing $H^*$ into
  $H^*_S=H^*-A$ and $A$. Hence, $A$ is an unbalanced circuit.

{If $A$ intersects $P_1$ in more than one vertex, say $A$ intersects
  $P_1$ in distinct vertices $u_1, u_2$, then we can decompose $H^*$
  as follows.  One subgraph, denoted by $B$, is a union of a subpath $P_1'$
  of $P_1$ between $u_1$ and $u_2$, and of every component of $H^*-P_1-P_2$ that intersects $P_1'$ and does not intersect $(P_1\cup P_2)-P_1'$, and of a $u_1$-$u_2$ path of $A$
  chosen in such a way that $B$ has even number of negative edges
  (since $A$ is unbalanced, this is possible). The second subgraph is
  defined as $H^*_S = H^*-E(B)$.  Again, we find a contradictory
  decomposition of $H^*$. Thus $A$ intersects $P_1$ in at most one
  vertex. Similarly, $A$ intersects $P_2$ in at most one vertex. As
  $H^*$ is $2$-connected, $A$ intersects both $P_1$ and $P_2$ in
  exactly one vertex (and these vertices are distinct).}
  
  Now we are ready to define the cover of $H$.  First, we define a set
  of unbalanced circuits $S$ and a closed trail $T$ such that
  $S\cup T=H$. Let $L$ be a set of loops whose intersection with
  $V(P_1\cup P_2)$ is nonempty.  The set $S$ contains all loops of $H$
  and all components of $H-P_1-P_2-L$ that are unbalanced
  circuits. Note that any two unbalanced circuits of $S$ that are not
  loops are vertex-disjoint, since they arise from two different
  components of $H^*-P_1-P_2$. The closed trail $T$ starts at $v_1$
  and traverses every edge of $H$ that is not contained in any circuit
  from $S$ in such a way that edges of $P_2$ are the last edges of the
  trail.  We label the elements of $S$ according to the order in which
  $T$ visits them for the first time, starting with the unbalanced
  loop $e_1$. Thus we obtain an ordered $n$-tuple of unbalanced
  circuits $(C_1=e_1, C_2, \dots, C_n)$. For $C_i$, define $u_i$ to be
  the first vertex of $C_i$ visited by $T$. We define barbells
  covering $H$ as follows.  Take $C_i$, $C_{i+1}$ and the segment of
  $T$ between $u_i$ and $u_{i+1}$, for $i \in \{1, 2, \dots,
  n-1\}$. Moreover, take the barbell containing $C_n$, $C_1$ and the
  segment of $T$ after $u_n$. Next, we will prove that the constructed
  elements are, indeed, barbells.

{{Note that between any two occurrences of the same vertex on the
  trail $T$ there is a vertex that is incident to a loop. }}
  Thus $u_i$-$u_{i+1}$
  subtrails of $T$ contain no circuits, for
  $i \in \{1, 2, \dots, n\}$. If two unbalanced circuits, say $C_i$
  and $C_j$, intersect, then due to the definition of $S$ at least one
  of them is a loop, for $i,j\in\{1,2,\ldots,n\}$. Moreover, the trail
  segments do not internally intersect the unbalanced circuits,
  because the components of $H-P_1-P_2$ intersect $P_1$ only in one
  vertex and the segments of $T$ containing edges of $P_1$ and $P_2$
  are separated by the loop $e_2$. Therefore the constructed cover
  consists of barbells.  It covers the edges of circuits in $S$, which
  contains all loops, exactly twice, while the edges in $T$ are
  covered once. Thus the cover satisfies the statement of the lemma
  which contradicts the fact that $H$ is a counterexample.
\end{proof}


\section{Proof of Theorem~\ref{thm:main}}\label{sec:proof}

In this section, we follow ideas of~\cite{CLLZ} to provide the
framework for the proof of Theorem~\ref{thm:main}. We will show a
straightforward proof of the bound $4\cdot |E(G)|-5/3\cdot \epsilon_N$
instead of $11/3 \cdot |E(G)| - 5/3\cdot \epsilon_N$. Further, we will
analyse the straightforward proof. We identify the places that require
an improvement in order to prove Theorem~\ref{thm:main}, and postpone
the technical details to Section~\ref{sec:BC}. The following is an
easy observation.

\begin{lemma}\label{l:minsig}
  Let $\sigma$ be a minimum signature of $(G,\sigma)$. Then for every
  edge-cut of $(G,\sigma)$ the number of negative edges does not
  exceed the number of positive ones.
\end{lemma}

Given a set $\mathcal E$ of vertex-disjoint Eulerian subgraphs of a
signed graph $(G,\sigma)$, \emph{connecting $\mathcal E$ into a tree
  of Eulerian graphs} means taking the disjoint union of the graphs in
$\mathcal E$ and adding the minimum number of edges of $G$ needed to get a tree of Eulerian graphs.

We are going to prove Theorem~\ref{thm:main}. Thus, let $(G,\sigma)$ be a
flow-admissible signed graph; without loss of generality, we assume
that it is connected and that $\sigma$ is a minimum signature. Let $X$
be the set of negative edges of $(G,\sigma)$. If $X=\emptyset$, the
result follows by Theorem~\ref{BJJ}. Furthermore, we may assume that
$|X| \geq 2$ for if $|X|=1$, then $(G,\sigma)$ is not
flow-admissible.

Let $B$ be the set of such edges $b$ of $(G,\sigma)$ that $G-b$ has
two components, each of them being unbalanced. Note that since
$(G,\sigma)$ is flow-admissible, $B$ is the set of all bridges of
$G$. Moreover, as $\sigma$ is minimum, by Lemma~\ref{l:minsig}, we
have $B\cap X=\emptyset$. Let $S$ be the set of such edges $s$ of
$(G,\sigma)$ that there exists a 2-edge-cut $\{s,t\}$, where $t\in
X$. Note that by Lemma~\ref{l:minsig} and the minimality of $\sigma$,
$S\cap X=\emptyset$. Furthermore, $S$ is exactly the set of bridges in
$G-X-B$. This implies the following claim.

\begin{claim}\label{cl:1}
  $G-X-B-S$ is a bridgeless balanced graph.\qed
\end{claim}

By Claim~\ref{cl:1} and Theorem~\ref{BJJ}, $G-X-B-S$ has a signed circuit cover
of length at most $5/3\cdot |E(G-X-B-S)|$. We will construct a
collection of signed circuits $\mathcal{C}$ that covers the edges of
$X\cup B\cup S$ with total length at most $2\cdot |E(G)|$. This will
prove Theorem~\ref{thm:main}.

Let $T$ be a spanning tree of the connected graph $G-X$, and we denote
$T\cup X$ by $(G',\sigma')$, where $\sigma'$ is the restriction of
$\sigma$ to $G'$. Note that $\sigma'$ may not be a minimum signature
of $(G',\sigma')$. Let $X', B', S'$ be defined on $(G',\sigma')$ in
the same way as $X,B,S$ on $(G,\sigma)$, respectively.

\begin{claim}\label{cl:2}
  $X'=X$, $B'\supseteq B$ and $S'\supseteq S$.\qed
\end{claim}

This straightforward claim guarantees that it is enough to find a
collection of signed circuits of $(G',\sigma')$ that covers
$X'\cup B' \cup S'$. To finish the proof of Theorem~\ref{thm:main}, we
will prove the following lemma.

\begin{lemma}\label{l:main}
  Let $(G',\sigma')$ be a signed graph such that $G'-X'$ is a spanning
  tree of $G'$, where $X'$ denotes the set of negative edges of
  $G'$. Let $B'$ be the set of bridges of $G'$ which separate two unbalanced subgraphs, and let $S'$ be the
  set of such edges of $(G',\sigma')$ that there exists a $2$-edge-cut $\{s',t'\}$, where $t'\in X'$.
  If $|X'| \ge 2$, then there exists a
  collection of signed circuits $\mathcal{C'}$ of $(G',\sigma')$ that
  covers the edges of $X'\cup B'\cup S'$ with total length at most
  $2\cdot |E(G')|$. Moreover, $\mathcal{C'}$ covers every negative
  loop of $G'$ exactly twice.
\end{lemma}
\begin{proof}
  For the sake of a contradiction, suppose that $(G',\sigma')$ is a
  counterexample with the minimum number of edges. Clearly,
  $(G',\sigma')$ has no balanced loops. {{Also $(G',\sigma')$
  has no vertex of degree 1 as the incident bridge does not belong to $X'\cup B'\cup S'$.}}

  We prove that $G'$ is $2$-connected. Suppose it is not. Let $G'_1$
  be an end-block of $G'$ incident with a cut-vertex $v$, and let
  $G'_2$ be obtained by removing all vertices of $G'_1$ except
  $v$. Each of these graphs inherits a signature from $G'$ that will
  be omitted from the notation. Since $G'$ contains no vertex of
  degree 1, both $G'_1 - T$ and $G'_2-T$ are non-empty and thus $G'_1$ and $G'_2$ are unbalanced.  For
  $i\in\{1,2\}$, we add a negative loop $e_i$ to $G'_i$ at $v$ and denote
  the resulting graph by $G''_i$. The signed graphs $G''_1$ and
  $G''_2$ satisfy the conditions of the lemma, because they have fewer
  edges than $G'$ and each of them has at least two negative
  edges. Note that $e_1$ and $e_2$ must be covered by barbells, and
  that each of them is covered exactly twice. {{To obtain the sought collection of signed circuits for $(G',\sigma')$, we combine the found signed circuit covers of $G''_1$ and $G''_2$ by merging the barbells
  containing loops $e_1$ and $e_2$. }}
   Taking into account that the
  $e_1$ and $e_2$ are covered exactly twice, we can bound the length
  of the cover of $(G',\sigma')$ by
  $$
  2 \cdot |E(G''_1)|-2+2 \cdot |E(G''_2)|-2=2 \cdot |E(G'_1)|+2 \cdot
  |E(G'_2)|=2 \cdot |E(G')|,
  $$
  a contradiction.

  Thus, $G'$ is 2-connected as claimed. This means that $G'$ is either
  a vertex with loops incident with it, or an edge with loops incident
  to its end-vertices, or a bridgeless graph with at least two
  vertices. In the first two cases, the required collections are easy
  to find directly. It follows that $G'$ is bridgeless (with
  $\size{V(G')}\geq 2$) and hence $B'=\emptyset$, so it suffices to
  cover the edges of $X'\cup S'$.

  For $x\in X'$, let $C_x$ be the unique circuit of $T\cup x$. For
  $A\subseteq X'$, let $C_A$ be the symmetric difference of all
  circuits $C_a$ for $a\in A$. Given $x\in X'$, let $S'_x$ be the set
  of such edges $s$ of $S'$ that belong to a 2-edge-cut $\{s,x\}$ of
  $(G',\sigma')$. For $A\subseteq X'$, let $S'_A$ be the union of
  $S'_a$ over all $a\in A$. Since $X'$ contains no $2$-edge-cut,
  $S'_x\cap S'_y=\emptyset$ for $x\neq y$.

  \begin{claim}\label{cl:3}
    For $A\subseteq X'$, $C_{A}$ contains every edge of
    $A\cup S'_A$.\qed
  \end{claim}

  Since the intersection of two paths in a tree is either a path
  (possibly trivial) or the empty graph, the following holds.

  \begin{claim}\label{cl:4}
    If $|A|=2$, then $C_A$ is either a balanced circuit, or a short
    barbell, or a union of two vertex-disjoint unbalanced circuits.\qed
  \end{claim}

  We will distinguish three cases based on the number of loops of $G'$.

  \begin{xcase}{Case A}
    The graph $G'$ contains at least two loops.
  \end{xcase}

  Consider $C_{X'}$ and note that it contains all loops of $G'$ as
  well as $X'\cup S'$ according to Claim 3. We connect $C_{X'}$
  into a tree of Eulerian graphs $H$. If $H$ contains an even leaf
  balloon $F$, then by Corollary~\ref{cor:eventree}, $F$ can be
  covered by a signed circuit cover of total length at most
  $4/3\cdot |E(F)|$. Thus we may consider the signed graph $H-E(F)$,  which has fewer edges than $H$ (and at least two negative loops). {{We repeat this process until we obtain a signed graph $H^*$ that does not contain any even leaf
  balloon. By Lemma~\ref{l:gentree}, $H^*$ admits a signed circuit cover
  $\mathcal{C}$ of length at most $2\cdot |E(H^*)|$ such that every loop
  of $H^*$ is covered exactly twice. The combination of this cover with the ones found for even leaf balloons provides a signed circuit cover of $G'$ that satisfies the lemma, which is a contradiction with the choice of $G'$.}}

  \begin{xcase}{Case B}
    The graph $G'$ contains exactly one loop.
  \end{xcase}

  Let $\ell\in X'$ be the unique loop of $G'$. Suppose first that
  $|X'|$ is even. Connecting $C_{X'}$ into a tree of Eulerian graphs
  $H$ and applying Lemma~\ref{l:eventree}, we find a collection of
  signed circuits of total length at most
  $2\cdot |E(H)|$ covering $C_{X'}$ such that $\ell$ is covered twice. By Claim~\ref{cl:3},
  this cover of $X'\cup S'$ provides a counterexample.

  Thus, $|X'|$ is odd. Let $a$ and $b$ be distinct negative edges
  different from $\ell$. Consider $C_{X'-\{a\}}$, which has an even
  number of negative edges, and connect it into a tree of Eulerian
  graphs $H$. By Corollary~\ref{cor:eventree}, $H$ admits a collection of
  signed circuits $\mathcal{C}$ of total length at most
  $4/3\cdot |E(H)|$ that covers the edges of $C_{X'-\{a\}}$ and has
  width at most $2$. We now augment $\CC$ by another signed circuit
  $C_2$ to cover $a$ and cover $\ell$ for a second time if necessary.

  If $\mathcal{C}$ covers $\ell$ once, then we find $C_2$ by
  connecting $C_{\{\ell,a\}}$ into a barbell. If $\mathcal{C}$
  covers $\ell$ twice, then we apply Claim~\ref{cl:4} and define $C_2$
  either as $C_{\{a,b\}}$ (if it is a balanced circuit or a short
  barbell), or as a barbell obtained by joining the components of
  $C_{\{a,b\}}$ (if it is a disjoint union of unbalanced circuits). In
  either case, the length of $C_2$ is at most $|E(G')|$, and adding
  $C_2$ to the collection $\mathcal{C}$, we obtain a signed circuit
  cover of ${X'\cup S'}$ of total length at most
  $(2+1/3)\cdot |E(G')|$ covering $\ell$ twice, a contradiction {{(if we consider the statement of Lemma~\ref{l:main} where $2\cdot |E(G')|$ is replaced by $(2+1/3)\cdot |E(G')|$).}}

  In Section~\ref{sec:B}, we refine the analysis as follows. We show
  how to choose the edges $a$ and $b$, and one of the three signed
  circuit covers of $C_{X'-\{a\}}$ given by Lemma~\ref{l:eventree}, in
  such a way that total length of the cover is at most
  $2\cdot |E(G')|$. The only case when this is not possible is when
  $G'$ is homeomorphic to a certain graph. In the latter case, it is
  easy to find the necessary cover directly.

  \begin{xcase}{Case C}
    The graph $G'$ contains no loops.
  \end{xcase}

  Suppose first that $|X'|$ is even. Connecting $C_{X'}$ into a tree
  of Eulerian graphs $H$ and using Corollary~\ref{cor:eventree}, we
  find a collection of signed circuits of total length at most
  $4/3\cdot |E(H)|$ covering $C_{X'}$, which covers $X'\cup S'$ by
  Claim~\ref{cl:3} and leads to a contradiction.

  Thus, $|X'|$ is odd. We apply an argument similar to the one used in
  Case B. Let $a$ and $b$ be distinct negative edges. Connect
  $C_{X'-\{a\}}$ into a tree of Eulerian graphs $H$ and note that
  $\negatx H$ is even. By Corollary~\ref{cor:eventree}, $H$ admits a
  collection of signed circuits $\mathcal{C}$ of total length at
  most $4/3\cdot |E(H)|$ covering the edges of $C_{X'-\{a\}}$. Using
  Claim~\ref{cl:4}, we find a signed circuit $C_2$ containing
  $C_{\{a,b\}}$ (of length at most $\size{E(G')}$), and we add it to
  $\CC$. The resulting collection $\mathcal{C}\cup\Setx{C_2}$
  covers ${X'\cup S'}$ and its length is at most
  $(2+1/3)\cdot |E(G')|$, which is a contradiction {{(if we consider the statement of Lemma~\ref{l:main} where $2\cdot |E(G')|$ is replaced by $(2+1/3)\cdot |E(G')|$).}}

  In Section~\ref{sec:C}, we refine the analysis as follows. We show
  how to choose the edges $a$ and $b$ in such a way that the length of
  $C_{\{a,b\}}$ is at most $2/3\cdot\size{E(G')}$. We characterise the
  case when this is not possible, and show how to find the cover
  explicitly in this case.
\end{proof}


\section{Finishing Cases B and C}\label{sec:BC}

In this section, we finish the proof of Lemma~\ref{l:main}, retaining
the notation introduced in Section~\ref{sec:proof}. We will first
introduce some further notation and a lemma that will be used to
finish the arguments of Case B and Case C. Recall that $(G',\sigma')$
is a minimal counterexample to Lemma~\ref{l:main} and has been found
to be 2-connected and bridgeless {{with at least two vertices}}, and to contain at most one
loop. Furthermore, $X'$ is the set of negative edges of $G'$, and
$G'-X'$ is a spanning tree $T$ of $G'$.

For $A\subseteq X'$, let $D_A$ be the graph obtained as a union of
$C_a$ over all $a\in A$. Note that $C_x=D_{\{x\}}$ for any $x\in
X'$. Due to the construction of $G'$, $D_A-A$ is a forest and has the
same number of components as $D_A$.

In the following lemma we will show how to extend a 2-connected graph
$D_A$ by adding a suitable edge $x$ to $A$.

\begin{lemma}\label{l:D}
  Let $A\subseteq X'$ be such that $D_A$ is $2$-connected. If $X'-A$ contains no loop, then either
  \begin{itemize}
  \item[(i)] $A=X'$ and $D_A=G'$, or
  \item[(ii)] there exists $x\in X'-A$ such that $D_A\cap C_x$ is a non-trivial path and hence $D_{A\cup\{x\}}$ is $2$-connected, or
  \item[(iii)] $A$ contains only loops.
  \end{itemize}
\end{lemma}

\begin{proof}
  Since $D_A$ is connected, $D_A-A$ is a subtree of $T$. If $D_A$
  contains a single vertex, then $A$ contains only loops and (iii)
  holds. Therefore, $D_A$ contains at least two vertices. If
  $D_A = G'$, then (i) holds.
  Assume thus that there is an edge $e$ in
  $E(G')-E(D_A)$. {{If every edge of $E(G')-E(D_A)$ belongs to $X'$, then $T\subseteq D_A$, and thus (ii) holds. Hence we may assume that $e\in T$. Since $G'$ is connected, we may further choose $e$ whose end-vertex belongs to $V(D_A)$.}} We may assume that exactly one end-vertex of $e$ (say
  $v$) belongs to $D_A$, since otherwise $e\in X'$. Consider a
  component $K$ of $T-D_A$ that contains $e$. Since $G'$ is
  2-connected, there is an edge $x$ connecting a vertex from $K-\{v\}$
  to a vertex from $V(G')-V(K)$.  Observe that $x\in X'$, because $T$
  is a tree. It is easy to see that $D_A\cap C_x$ is a non-trivial
  path.
\end{proof}

{{In proofs of both Claim B and Claim C, we will need to define certain
barbells. The unbalanced circuits of the barbells can be defined
easily using elements of $X'$, since $C_x$ is a uniquely defined
elementary circuit of $G'$ for any $x\in X'$. If $C_x$ and $C_y$
are disjoint, then we define the barbell $B_{x,y}^Y$, where $Y\subseteq X'$, as follows.
We contract $C_x$ and $C_y$ into two new vertices $v_x$ and $v_y$.
Positive edges of the graph still form a tree.
Consider the path $P_{x,y}$ consisting of edges of $T$ that connects $v_x$ and $v_y$. If the symmetric difference of
$P_{x,y}$ and all elementary circuits containing edges from $Y$ is a path, denoted by $P^Y_{xy}$, then we set $B_{x,y}^Y=P_{x,y}^Y \cup C_x \cup C_y$.}}


\subsection{Case C}\label{sec:C}
Recall that the graph $G'$ has odd number of negative edges (and more
than one negative edge) and no loops. We want to find edges
$a,b \in X$ such that $C_{\{a,b\}}$ can be covered by a signed circuit
of length at most $2/3 \cdot |E(G')|$. If this is not possible, we
want to find a cover of $G'$ that satisfies Lemma~\ref{l:main}.

\begin{lemma}\label{l:3int}
  Suppose that there are $x,y,z \in X'$ such that
  $C_x \cap C_y \neq \emptyset$, $C_x \cap C_z \neq \emptyset$, and
  $C_y \cap C_z \neq \emptyset$. Then there are two edges $a,b \in X'$
  such that $C_{\{a,b\}}$ is either a balanced circuit or a short
  barbell of length at most $2/3 \cdot |E(G')|$.
\end{lemma}
\begin{proof}
  By Claim~\ref{cl:4}, the subgraphs $C_{\{x,y\}}$, $C_{\{x,z\}}$, and
  $C_{\{y,z\}}$ are balanced circuits or short barbells. Each edge $e$
  of $D_{\{x,y,z\}}$ belongs to one, two or three circuits among
  $C_x$, $C_y$, and $C_z$. If $e$ belongs to exactly one of the
  circuits, say $C_x$, then $e \not \in C_{\{y,z\}}$. If $e$ belongs
  to two of the circuits, say $C_x$ and $C_y$ (and possibly $C_z$),
  then $e \not \in C_{\{x,y\}}$. Thus
  $C_{\{x,y\}} \cap C_{\{x,z\}} \cap C_{\{y,z\}}$ has no edge, and
  $|E(C_{\{x,y\}})|+ |E(C_{\{x,z\}})| + |E(C_{\{y,z\}})| \leq 2\cdot
  |E(G')|$. The lemma follows.
\end{proof}

Let $x\in X'$. Since $|X'|\geq 3$, by Lemma~\ref{l:D} there is
$y\in X$ such that $C_x\cap C_y$ is a non-trivial path. Again by
Lemma~\ref{l:D}, there exists $z$ such that $C_z\cap D_{\{x,y\}}$ is a
non-trivial path. Due to Lemma~\ref{l:3int} we may assume, without
loss of generality, that $C_x \cap C_z = \emptyset$. An example of the
graph $D_{\{x, y, z\}}$ is depicted in Fig.~\ref{fig}.

\begin{figure}[htp]
  \center
  \includegraphics{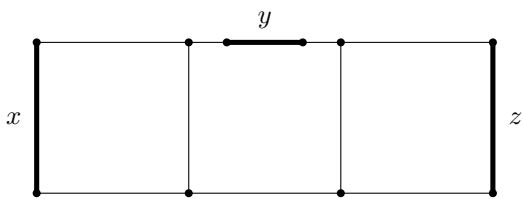}
  \caption{Possible configuration of graph $D_{\{x, y, z\}}$. As in
    the subsequent figures, the edges of $T$ are represented by thin
    lines and the edges of $X'$ are shown by thick lines.}
  \label{fig}
\end{figure}

If $G'= D_{\{x, y, z\}}$, then we can cover the edges of $G'$ by two
barbells $B_{x,z}^{\emptyset}$ and $B_{x,z}^{\{y\}}$. This cover has
the property stated in Lemma~\ref{l:main}. Therefore, we may assume
that $G'\neq D_{\{x, y, z\}}$. By Lemma~\ref{l:D} there is an edge
$w\in X'-\{x,y,z\}$ such that $C_w \cap D_{\{x, y, z\}}$ is a
non-trivial path. If some three of the circuits $C_x$, $C_y$, $C_z$,
$C_w$ have pairwise non-empty intersection, then we can apply
Lemma~\ref{l:3int}.

{{First, let us assume that one of
the circuits}} $C_x$, $C_y$, $C_z$, $C_w$, say $C_y$, intersects all the
other circuits. {{Due to Lemma~\ref{l:3int}, $C_w\cap C_x =\emptyset$ and $C_w\cap C_z =\emptyset$. Therefore, either $C_w\cap C_y \subseteq B_{xz}^{\emptyset}$ or $C_w\cap C_y \subseteq B_{xz}^{\{y\}}$. We deal with the first case, which is depicted in Fig.~\ref{figa}. The latter case is similar.}}

\begin{figure}[htp]
  \center
  \includegraphics{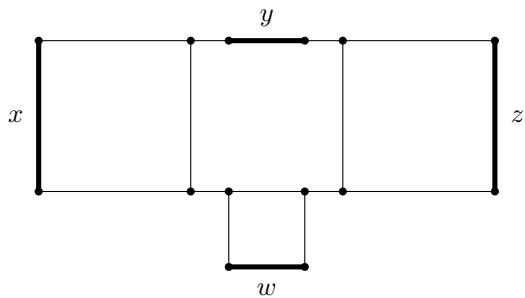}
  \caption{Possible configuration of $D_{\{x, y, z, w\}}$ in case $C_y$ intersects all other circuits.}
  \label{figa}
\end{figure}

The barbells $B_{x,z}^{\{y\}}$, $B_{x,w}^{\emptyset}$,
$B_{z,w}^{\emptyset}$ do not cover any edge of $G'$ more than twice in
total, therefore one of them has length at most $2/3 \cdot |E(G')|$.

{{It remains to consider}} that none of the circuits $C_x$, $C_y$, $C_z$,
$C_w$ intersects all the other circuits.  Without loss of generality
suppose that $C_x \cap C_y \neq \emptyset$,
$C_y \cap C_z \neq \emptyset$, and $C_z \cap C_w \neq \emptyset$. An
example of the graph $D_{\{x, y, z, w\}}$ is depicted in
Fig.~\ref{fig2}.
\begin{figure}[htp]
  \center
  \includegraphics{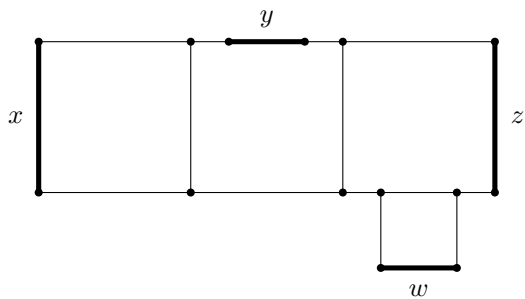}
  \caption{Possible configuration of $D_{\{x, y, z, w\}}$ in case that none of $C_x, C_y, C_z$ and $C_w$ intersects all
other circuits.}
  \label{figb}
\end{figure}

Note that $C_{\{x,y\}}$, $C_{\{y,z\}}$, and $C_{\{z,w\}}$ are balanced circuits and no edge is contained in all three of them. Therefore, one of them has length at most $2/3 \cdot |E(G')|$, which gives the desired choice of $a$ and $b$.


\subsection{Case B}\label{sec:B}

Recall that in Case B, the graph $G'$ is 2-connected,
$\negat{G'}\geq 3$ is odd, and $G'$ contains one (unbalanced) loop
$\ell$.  A \emph{solution} $\mathcal{S}$ in $G'$ is a quadruple
$(a,b, C_1, C_2)$ such that $a, b \in X'$, $C_1$ is a signed circuit
that contains $C_{\{\ell,a\} }$, and $C_2$ is a signed circuit that
contains $C_{\{a,b\}}$. The \emph{cover corresponding to the solution
  $(a,b, C_1, C_2)$} is the collection $\Setx{C_1,C_1,C_2}$.

\begin{lemma}\label{l:sol}
  There are no $k$ solutions in $G'$ such that the union of covers
  corresponding to these $k$ solutions has width at most $2k$.
\end{lemma}

\begin{proof}
  For the sake of a contradiction, assume that the $k$ solutions
  exist. Let $\mathcal{S}=(a,b,C_1,C_2)$ be one of them. Consider
  $C_{X'-\{a\}}$, which has even number of negative edges, and connect
  it into a tree of Eulerian graphs $H$. By Lemma~\ref{l:eventree},
  there are three weak signed circuit covers ${\cal C}_1$,
  ${\cal C}_2$, ${\cal C}_3$ of $H$ that cover edges of $C_{X'-\{a\}}$
  with total width at most $4$. Moreover, ${\cal C}_1$ covers the loop
  $\ell$ exactly twice, and each of ${\cal C}_2$ and ${\cal C}_3$
  covers $\ell$ exactly once. We define three collections of signed
  circuits as follows: ${\cal C}_1 \cup \Setx{C_2}$,
  ${\cal C}_2 \cup \Setx{C_1}$ and ${\cal C}_3 \cup \Setx{C_1}$.

  We follow this procedure for all $k$ solutions to obtain $3k$
  collections that cover edges of $X'\cup B'\cup S'$. The total width
  of these $3k$ collections is obtained as the sum of total width of
  covers corresponding to the $k$ solutions, which is at most $2k$ by
  the assumption, and the total width of $3k$ covers obtained from $k$
  uses of Lemma~\ref{l:eventree}, which is at most $4k$; hence the
  total width of the $3k$ collections thus defined is at most $6k$. We
  conclude that one of the collections must have length at most
  $2 \cdot |E(G')|$. Since each of the collections covers the loop
  $\ell$ exactly twice, we obtained a contradiction with the
  assumption that $G'$ is a counterexample to Lemma~\ref{l:main}.
\end{proof}

\begin{lemma}\label{l:l}
  There are no two edges $x,y \in X'$ such that
  $C_x \cap C_y \cap C_\ell \neq \emptyset$.
\end{lemma}
\begin{proof}
  Suppose that edges $x,y \in X'$ have the stated property. We
  construct two solutions of width $4$:
  \begin{align*}
    \mathcal{S}_1&=(x,y,C_{\{\ell,x\}}, C_{\{x,y\}}),\\
    \mathcal{S}_2&=(y,x,C_{\{\ell,y\}}, C_{\{y,x\}}).
  \end{align*}
  These solutions provide a contradiction with Lemma~\ref{l:sol}.
\end{proof}

Choose $x$ such that the graph $D_{\{\ell,x\}}$ is connected and note
that $D_{\{\ell,x\}}$ is 2-connected. Since $|X'| \ge 3$, by
Lemma~\ref{l:D}, there exists an edge $y \in X'$ such that
$D_{\{\ell,x,y\}}$ is 2-connected. By Lemma~\ref{l:l},
$C_\ell\cap C_x \cap C_y=\emptyset$. An example showing
$D_{\{\ell,x,y\}}$ is given in Fig.~\ref{fig2}.

\begin{figure}[htp]
  \center
  \includegraphics{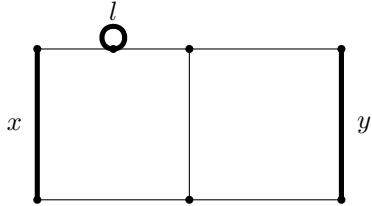}
  \caption{Possible configuration of $D_{\{\ell,x,y\}}$.}
  \label{fig2}
\end{figure}

If $G'=D_{\{\ell,x,y\}}$, then a collection of signed circuits
required by Lemma~\ref{l:main} that provides a contradiction is
$\Setx{C_{\{\ell,x\}}, B_{\ell,y}^\emptyset}$. Otherwise, by
Lemma~\ref{l:D}, there exists an edge $z\in X'$ such that
$C_z \cap D_{\{\ell,x,y\}}$ is a non-trivial path $P$.  We will
discuss the possible positions of $P$ with respect to the segments of
$D_{\{\ell,x,y\}}$.

By Lemma~\ref{l:l}, $C_z\cap C_x \cap C_\ell = \emptyset$. Suppose
first that $P$ intersects $C_x-C_y$. Then there exist two
edge-disjoint paths $P_1$ and $P_2$ in $D_{\{\ell,x,y,z\}}-C_y-C_z$,
$P_1$ connecting $\ell$ to $C_y$ and $P_2$ connecting $\ell$ to $C_z$.
Then the following collections of solutions satisfy
Lemma~\ref{l:sol}, which is a contradiction. Define $C^*$ to be either a barbell containing
$C_{\{y,z\}}$ and not containing $V(\ell)$ (for
$E(C_y)\cap E(C_z)=\emptyset$) or a balanced circuit $C_{\{y,z\}}$
(otherwise).  We set $S_1=(y,z, C_{\{\ell,y\}} \cup P_1,C^*)$ and
$S_2=(z,y, C_{\{\ell,z\}} \cup P_2,C^*)$. Note that when $C_{\{y,z\}}$
is a circuit, the edges of $C_y \cap C_z$ are covered by both barbells
but not by $C^*$.

Suppose finally that $P$ does not intersect $C_x-C_y$. Then
$P\subseteq C_y$ and there are three possible subcases:
\begin{enumerate}[(i)]
\item $P\subseteq C_y - C_x$ and $B_{l,z}^{\emptyset}$ does not contain any
edge of $C_x\cap C_y$,
\item $P\subseteq C_y - C_x$ and
$B_{l,z}^{\emptyset}$ does contain an edge of $C_x\cap C_y$,
\item $P\subseteq C_x\cap C_y$.
\end{enumerate}
Examples for these subcases are depicted in Fig.~\ref{fig3}.

\begin{figure}[htp]
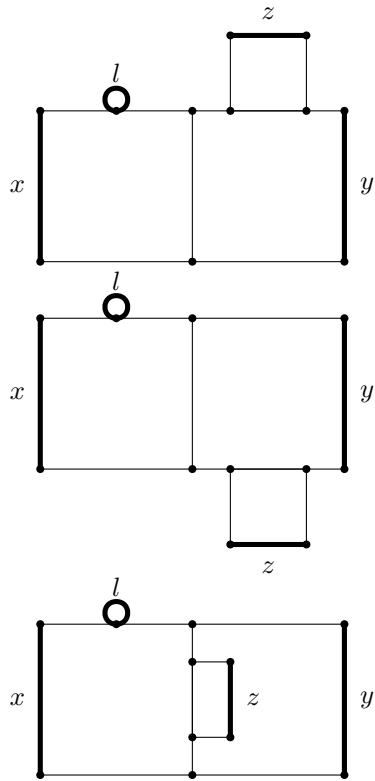

  \center
  \includegraphics{sb-4}\\
  \includegraphics{sb-3}\\
  \includegraphics{sb-5}
  \caption{Possible configurations of $D_{\{\ell,x,y,z\}}$.}
  \label{fig3}
\end{figure}

For each of the subcases, there are three solutions such that the
corresponding covers have total width at most $6$. These solutions are
given in Table~\ref{tab:solutions}. This contradiction to
Lemma~\ref{l:sol} finishes the discussion of Case B and hence also the
proof of Lemma~\ref{l:main}.\qed

\begin{table}
  \centering
  \begin{tabular}{r|l}
    case & solutions\\\hline
    (i) & $(x,y,C_{\{\ell,x\}},C_{\{x,y\}}),
          (y,z, B_{\ell,y}^{\emptyset},C_{\{y,z\}}),
          (z,x,B_{\ell,z}^{\{x,y\}},B_{x,z}^{\emptyset})$\\
    (ii) & $(x,y,C_{\{\ell,x\}},C_{\{x,y\}}),
           (y,z, B_{\ell,y}^{\emptyset},C_{\{y,z\}}),
           (z,x,B_{\ell,z}^{\{x\}},B_{x,z}^{\{y\}})$\\
    (iii) & $(x,y,C_{\{\ell,x\}},C_{\{x,y\}}),
            (z,x, B_{\ell,z}^{\emptyset},C_{\{z,x\}}),
            (z,y,B_{\ell,z}^{\{x\}},C_{\{z,y\}})$
  \end{tabular}
  \caption{Solutions for subcases (i)--(iii) when $P$ does not
    intersect $C_x-C_y$.}
  \label{tab:solutions}
\end{table}


\section{Alternative bounds}\label{sec5}

In this section we provide several alternative bounds on the length of a shortest signed circuit cover of a flow-admissible signed graph.
We compare these bounds to results of Cheng et al. \cite{CLLZ}. We will follow notation from Section~\ref{sec:proof}.

The graph $G-X-B-S$ has at most $|E(G)|-|X|$ edges. The graph $G'$ has at most
$|V(G)|-1+|X|$ edges.
By Theorem~\ref{BJJ}, we can cover $G-X-B-S$ with a circuit cover of length at most
$5/3 \cdot (|E(G)|-|X|)$ and by Lemma~\ref{l:main} we can cover the edges
$X \cup B \cup S$ in $G'$ with signed circuits of total length at most
$2\cdot|V(G)|+2|X|-2$. Altogether, the constructed signed circuit cover has length at most
$5/3 \cdot |E(G)|+2|V(G)|+1/3\cdot |X| -2$, while the bound obtained by Cheng et al. \cite{CLLZ} is
$5/3 \cdot |E(G)|+3|V(G)|+4/3\cdot |X| -7$.
Due to minimality of $X$ and connectivity of $G$ we have $|E|\ge |X|+|V|-1$.
Using this inequality we obtain the bound $11/3 \cdot |E(G)|-5/3 \cdot |X|$ from Theorem~\ref{thm:main}.

For dense bridgeless graphs a result of Fan \cite{F} provides a better bound on shortest circuit covers of bridgeless graphs than
Theorem~\ref{BJJ}.  Fan showed that a bridgeless graph on $n$ vertices and $m$ edges can be covered with
circuits of total length at most $n+m-1$. This implies that $G-X-B-S$
can be covered by circuits of total length at most $|E(G)|-|X|+|V(G)|-1$.
Together with the cover of $X \cup B \cup S$ in $G'$ we obtain
a signed circuit cover of $G$ of length at most
$|E(G)|+3\cdot |V(G)|+|X| -3$, while the bound obtained by Cheng et al. \cite{CLLZ} is $|E(G)|+4\cdot |V(G)|+2\cdot |X| -8$.

\section*{Acknowledgements}
The first and the fourth author acknowledge support from projects GA17-04611S and GA14-19503S of
the Czech Science Foundation.

\noindent The work of the second and the third author was partially supported by the VEGA grant No. 1/0876/16 and by the APVV grant No. APVV-15-0220.

\noindent The work of the fourth author was partially supported by the project
LO1506 of the Czech Ministry of Education, Youth and Sports.


\end{document}